\documentclass[12pt]{amsart}

\usepackage{amssymb}

\usepackage{enumerate}

\usepackage[T1]{fontenc}

\makeatletter
\@namedef{subjclassname@2010}{%
  \textup{2010} Mathematics Subject Classification}
\makeatother

\newtheorem{thm}[subsection]{Theorem}
\newtheorem{prop}[subsection]{Proposition}
\newtheorem{lem}[subsection]{Lemma}

\theoremstyle{definition}
\newtheorem{defin}[subsection]{Definition}
\newtheorem{rem}[subsection]{Remark}
\newtheorem{exam}[subsection]{Example}

\newtheorem*{prob}{Problem}

\newcommand{\imax}{\underline{m}}

\DeclareMathOperator{\dist}{dist}
\DeclareMathOperator{\supp}{supp}

\begin{document}

%%%%% To ease editing, for IMPAN journals add:

\baselineskip=17pt

%%%%%%%%%%%%%%%%%%%%%%%%%%%%%%%%%%%

\title[{\L}ojasiewicz ideals]{{\L}ojasiewicz ideals\\ in Denjoy-Carleman classes}

\author[V. Thilliez]{Vincent Thilliez}
\address{Laboratoire Paul Painlev\'e, Math\'ematiques - B\^atiment M2\\
Universit\'e Lille 1\\
F-59655 Villeneuve d'Ascq Cedex, France}
\email{thilliez@math.univ-lille1.fr}

\date{}

\begin{abstract} The classical notion of {\L}ojasiewicz ideals of smooth functions is studied in the context of non-quasianalytic Denjoy-Carleman classes. In the case of principal ideals, we obtain a characterization of {\L}ojasiewicz ideals in terms of properties of a generator. This characterization involves a certain type of estimates that differ from the usual {\L}ojasiewicz inequality. We then show that basic properties of {\L}ojasiewicz ideals in the $\mathcal{C}^\infty$ case have a Denjoy-Carleman counterpart. 
\end{abstract}



\maketitle

\section{Introduction}
Let $\Omega$ be an open subset of $\mathbb{R}^n$, and let $\mathcal{C}^\infty(\Omega)$ be the Fr\'echet algebra of smooth functions in $\Omega$. Let $X$ be a closed subset of $\Omega$. An element $\varphi$ of $\mathcal{C}^\infty(\Omega)$ is said to satisfy the \emph{{\L}ojasiewicz inequality with respect to $X$} if, for every compact subset $K$ of $\Omega$, there are real constants $C>0$ and $\nu\geq 1$ such that, for any $x\in K$, we have 
\begin{equation}\label{ineq1}
\vert \varphi(x)\vert\geq C \dist(x,X)^\nu.
\end{equation}
For example, it is well-known that any real-analytic function satisfies the {\L}ojasiewicz inequality with respect to its zero set. 

An element of $\mathcal{C}^\infty(\Omega)$ is said to be \emph{flat on $X$} if it vanishes, together with all its derivatives, at each point of $X$. Denote by $\imax^\infty_X$ the ideal of functions of $\mathcal{C}^\infty(\Omega)$ that are flat on $X$. The following statement appears in {\cite[Section V.4]{Tou}} and establishes a connection between the {\L}ojasiewicz inequality and the behavior of ideals with respect to flat functions. 
\begin{thm}\label{Touger}
Let $\mathcal{I}$ be a finitely generated proper ideal in $\mathcal{C}^\infty(\Omega)$, and let $X$ be the zero set of $\mathcal{I}$. The following properties are equivalent:
\begin{enumerate}[$(A)$]
\item The ideal $\mathcal{I}$ contains an element $\varphi$ which satisfies the {\L}ojasiewicz inequality with respect to $X$.
\item $\imax^\infty_X\subset \mathcal{I} $.
\item $ \imax^\infty_X= \mathcal{I}\imax^\infty_X $.
\end{enumerate}
\end{thm}
A finitely generated ideal $\mathcal{I}$ satisfying the equivalent properties $(A)$, $(B)$, $(C)$ is called a 
\emph{{\L}ojasiewicz ideal}. A principal ideal is {\L}ojasiewicz if and only if condition $(A)$ holds for a generator $\varphi$ of the ideal. In the general case of a finitely generated ideal with generators $ \varphi_1,\ldots,\varphi_p$, one can take $\varphi=\varphi_1^2+\cdots+\varphi_p^2$. {\L}ojasiewicz ideals play an important role in the study of ideals of differentiable functions; see for instance \cite{Ris, Tho, Tou}. In particular, every closed ideal of finite type is {\L}ojasiewicz, whereas the converse statement is false.  

In the present paper, we study a possible approach to {\L}ojasiewicz ideals in non-quasianalytic Denjoy-Carleman classes $\mathcal{C}_M(\Omega)$. While several papers have already been devoted to the study of closed ideals in $\mathcal{C}_M(\Omega)$ (see for example \cite{Th1, Th2, Th3}), a suitable notion of {\L}ojasiewicz ideal is still lacking, even in the case of principal ideals. This is due to the fact that if we put $\imax_{X,M}^\infty = \imax_X^\infty\cap \mathcal{C}_M(\Omega) $ and $ \mathcal{I}=\varphi \mathcal{C}_M(\Omega) $, where $ \varphi$ is a given element of $ \mathcal{C}_M(\Omega)$, it turns out that the usual {\L}ojasiewicz inequality \eqref{ineq1} is not a sufficient condition for the inclusion $ \mathcal{I}\subset \imax_{X,M}^\infty $, let alone for the equality $ \imax_{X,M}^\infty = \mathcal{I}\imax_{X,M}^\infty$. Therefore, it is natural to ask for a characterization of both of these properties in terms of the generator $ \varphi$, in the spirit of the characterization given by Theorem \ref{Touger} in the $\mathcal{C}^\infty$ case. 

In the case of principal ideals, a suitable characterization will be obtained in Theorem \ref{main}. In the statement, the {\L}ojasiewicz inequality \eqref{ineq1} has to be replaced by a quite different property involving successive derivatives of $ 1/\varphi$, which will be shown to be equivalent to the obvious Denjoy-Carleman version of property $(C)$, that is, to the equality $ \imax_{X,M}^\infty = \mathcal{I}\imax_{X,M}^\infty $. We are also able to get an equivalence with a corresponding version of property $(B)$,  provided we consider the inclusion $ \imax_{X,M}^\infty\subset \mathcal{I} $ together with a mild extra requirement on the flat points of $\varphi$.  

In order to prove these results, one has to deal with the fact that the constructive techniques used by Tougeron in the classical $\mathcal{C}^\infty$ case do not seem applicable to the $\mathcal{C}_M$ setting. Thus, the main part of our proof of Theorem \ref{main} is actually based on a functional-analytic argument. Once the theorem is proven, we discuss several related properties showing that basic results of the $ \mathcal{C}^\infty$ case are extended in a consistent way. For instance, we show that our $\mathcal{C}_M$ {\L}ojasiewicz condition holds for closed principal ideals, and we also provide a non-closed example. 

\section{Denjoy-Carleman classes}
\subsection{Notation} For any multi-index $ J=(j_1,\ldots,j_n) $ of $ \mathbb{N}^n $, we always denote the length $ j_1+\cdots+j_n $ of $ J $ by the corresponding lower case letter $ j $. We put $ J!=j_1!\cdots j_n! $,
$ D^J=\partial^j/\partial x_1^{j_1}\cdots\partial x_n^{j_n} $ and $ x^J=x_1^{j_1}\cdots x_n^{j_n} $. We denote by $\vert \cdot\vert$ the euclidean norm on $\mathbb{R}^n$; balls and distances in $\mathbb{R}^n$ will always be considered with respect to that norm.

If $ a$ is a point in $ \mathbb{R}^n$, and if $ f$ is a smooth function in a neighborhood of  $a$, we denote by $ T_af $ the formal Taylor series of $ f$ at $a$, that is, the element of $ \mathbb{C}[[x_1,\ldots,x_n]]$ defined by
\begin{equation*}
T_af= \sum_{J\in \mathbb{N}^n}\frac{1}{J!}D^Jf(a)x^J.
\end{equation*}
The function $f$ is said to be \emph{flat} at the point $a $ if $T_af=0 $. 

\subsection{Some properties of sequences}\label{sequences}
Let $ M=(M_j)_{j\geq 0} $ be a sequence of real numbers satisfying the following assumptions:
\begin{equation}\label{norm}
\text{ the sequence } M \text{ is increasing, with } M_0=1,
\end{equation}
\begin{equation}\label{logc}
\text{ the sequence } M \text{ is \emph{logarithmically convex}}.
\end{equation}
Property \eqref{logc} amounts to saying that $ M_{j+1}/M_j $ is increasing. Together with \eqref{norm}, it implies
\begin{equation}\label{logprod}
M_jM_k\leq M_{j+k}\ \textrm{ for any } (j,k)\in\mathbb{N}^2. 
\end{equation}
We say that the \emph{moderate growth} property holds if there is a constant $ A>0$ such that, conversely, 
\begin{equation}\label{modg}
M_{j+k}\leq A^{j+k} M_jM_k\ \textrm{ for any } (j,k)\in\mathbb{N}^2. 
\end{equation}
We say that $M $ satisfies the \emph{strong non-quasianalyticity} condition if there is a constant $A>0 $ such that
\begin{equation}\label{snqa}
\sum_{j\geq k}\frac{M_j}{(j+1)M_{j+1}}\leq A \frac{M_k}{M_{k+1}} \textrm{ for any } k\in\mathbb{N}.
\end{equation}
Notice that property \eqref{snqa} is indeed stronger than the classical Denjoy-Carl\-eman non-quasianalyticity condition
\begin{equation}\label{dcqa}
\sum_{j\geq 0}\frac{M_j}{(j+1)M_{j+1}}<\infty.
\end{equation}
The sequence $M$ is said to be \emph{strongly regular} if it satisfies \eqref{norm}, \eqref{logc}, \eqref{modg} and \eqref{snqa}.

\begin{exam}\label{exgev}
Let $ \alpha $ and $\beta$ be real numbers, with $ \alpha> 0 $. The sequence $M$ defined by $ M_j=j!^\alpha(\ln(j+e))^{\beta j} $ is strongly regular. This is the case, in particular, for Gevrey sequences $ M_j=j!^\alpha $.
\end{exam}

With every sequence $ M $ satisfying \eqref{norm} and \eqref{logc} we also associate the function $ h_M $ defined by $ h_M(t)=\inf_{j\geq 0}t^jM_j $ for any real $ t>0 $, and $ h_M(0)=0 $. From \eqref{norm} and \eqref{logc}, it is easy to see that the function $h_M$ is continuous, increasing, and it satisfies $ h_M(t)>0$ for $t>0$ and $ h_M(t)=1 $ for $t\geq 1/M_1 $. It also fully determines the sequence $M$, since we have $ M_j=\sup_{t>0}t^{-j}h_M(t)$. 

\begin{exam} Let $M$ be defined as in Example \ref{exgev}, and put $\eta(t)=\exp(-(t\vert\ln t\vert^\beta)^{-1/\alpha}) $ for $t>0$. Elementary computations show that there are constants $a>0$, $b>0$ such that $ \eta(at)\leq h_M(t)\leq \eta(bt)$ for $t\to 0 $. 
\end{exam} 

A technically important consequence
of the moderate growth assumption \eqref{modg} is the 
existence of a constant $ \rho\geq 1 $, depending
only on $ M $, such that
\begin{equation}\label{hfunct2}
h_M(t)\leq \big(h_M(\rho t)\big)^2\text{ for any }t\geq 0.
\end{equation}
We refer to \cite{CC} for a proof that \eqref{norm}, \eqref{logc} and \eqref{modg} imply \eqref{hfunct2}.

\subsection{Denjoy-Carleman classes}\label{classes}
Let $\Omega$ be an open subset of $\mathbb{R}^n $, and let $M$ be a sequence of real numbers satisfying \eqref{norm} and \eqref{logc}. We define $ \mathcal{C}_M(\Omega)$ as the space of functions $ f$ belonging to $ \mathcal{C}^\infty(\Omega) $ and satisfying the following condition: for any compact subset $ K $ of $ \Omega $, one can find a real $ \sigma >0 $ and a constant $C>0 $ such that  
\begin{equation}\label{roum}
\vert D^Jf(x)\vert \leq C\sigma^j j!M_j\textrm{ for any }ÊJ\in \mathbb{N}^n\textrm{ and } x\in K.
\end{equation}

Given a function $ f$ in $ \mathcal{C}^\infty(\Omega) $, a compact subset $K$ of $\Omega$ and a real number $ \sigma>0 $, put
\begin{equation*}
\Vert f\Vert_{K,\sigma}=\sup_{x\in K,\ J\in \mathbb{N}^n}\frac{\vert D^J f(x)\vert}{\sigma^j j!M_j}. 
\end{equation*}
We see that $ f$ belongs to $ \mathcal{C}_M(\Omega)$ if and only if, for any compact subset $K$ of $\Omega$, one can find a real $ \sigma>0$ such that $\Vert f\Vert_{K,\sigma}$ is finite ($\Vert f\Vert_{K,\sigma}$ then coincides with the smallest constant $ C $ for which \eqref{roum} holds). The function space $ \mathcal{C}_M(\Omega)$ is called the \emph{Denjoy-Carleman class of functions of class $ \mathcal{C}_M $ in the sense of Roumieu} (which corresponds to $\mathcal{E}_{\{j!M_j\}}(\Omega)$ in the notation of \cite{Kom}). 

From now on, we will assume that the sequence $M$ is strongly regular. In particular, it satisfies \eqref{dcqa}, which implies that $ \mathcal{C}_M(\Omega)$ contains compactly supported functions.  We denote by $ \mathcal{D}_M(\Omega)$ the space of elements of $ \mathcal{C}_M(\Omega)$ with compact support in $ \Omega$. 

For the reader's convenience, we now recall some basic topological facts about $ \mathcal{C}_M(\Omega)$ and $ \mathcal{D}_M(\Omega)$, without proof (we refer  to \cite{Kom} for the details). With each Whitney 1-regular compact subset $K$ of $ \Omega$, and each integer $ \nu\geq 1 $, we associate the vector space $ \mathcal{C}_{M,K,\nu} $ of all functions $ f $ which are $ \mathcal{C}^\infty $-smooth on $K $ in the sense of Whitney, and such that $\Vert f\Vert_{K,\nu}<\infty $. 
Then $ \mathcal{C}_{M,K,\nu} $ is a Banach space for the norm $ \Vert \cdot\Vert_{K,\nu}$ and it can be shown that for $ \nu<\nu' $, the inclusion $ \mathcal{C}_{M,K',\nu}\hookrightarrow  \mathcal{C}_{M,K',\nu'} $ is compact. We define the Denjoy-Carleman class $ \mathcal{C}_M(K) $ as the reunion of all spaces $ \mathcal{C}_{M,K,\nu} $ with $ \nu \geq 1 $. Endowed with the inductive topology, $ \mathcal{C}_M(K) $ is a (DFS)-space (or \emph{Silva space}).  
Given an exhaustion $(K_j)_{j\geq 1} $ of $ \Omega$ by Whitney 1-regular compact subsets, the Denjoy-Carleman class $ \mathcal{C}_M(\Omega) $ can be identified with the projective limit of all (DFS)-spaces $ \mathcal{C}_M(K_j)$.  

Similarly, denote by $ \mathcal{D}_{M,K,\nu}$ the space of all functions $ f \in \mathcal{C}^\infty(\Omega)$ such that $ \supp f\subset K $ and $\Vert f\Vert_{K,\nu}<\infty $. Then $ \mathcal{D}_{M,K,\nu} $ is a Banach space and we have the following properties: for $ K\subset K' $, the space $ \mathcal{D}_{M,K,\nu} $ is a closed subspace of $ \mathcal{D}_{M,K',\nu} $, and for $ \nu<\nu' $, the inclusion $ \mathcal{D}_{M,K',\nu}\hookrightarrow  \mathcal{D}_{M,K',\nu'} $ is compact. For any integer $ \nu\geq 1 $, put $ \mathcal{D}_\nu=\mathcal{D}_{M,K_\nu,\nu} $, $ \Vert \cdot\Vert_\nu=\Vert \cdot\Vert_{K_\nu,\nu} $, and notice that we have $ \mathcal{D}_M(\Omega)=\bigcup_{\nu\geq 1} \mathcal{D}_\nu $ as a set. By the preceding remarks, we have a compact injection $ \mathcal{D}_\nu\hookrightarrow \mathcal{D}_{\nu+1} $. Thus, the space $ \mathcal{D}_M(\Omega) $ is another (DFS)-space for the corresponding inductive limit topology. 

\subsection{Some basic properties of $\mathcal{C}_M(\Omega)$}\label{basicprop}
Properties \eqref{norm} and \eqref{logc} of the sequence $M$ ensure that $ \mathcal{C}_M(\Omega) $ is an algebra containing the algebra of real-analytic functions, and that $\mathcal{C}_M$ regularity is stable under composition \cite{Rou}. This implies, in particular, the following invertibility property. 

\begin{lem}[\cite{Rou}]\label{invers}
If the function $f$ belongs to $ \mathcal{C}_M(\Omega)$ and has no zero in $ \Omega$, then the function $1/f$ belongs to $ \mathcal{C}_M(\Omega) $.
\end{lem}

It is also known  that the implicit function theorem holds within the framework of $ \mathcal{C}_M$ regularity \cite{Kom2}. Thus, $\mathcal{C}_M$ manifolds and submanifolds can be defined in the usual way.
 
The strong regularity assumption on $M$ ensures that suitable versions of Whitney's extension theorem and Whitney's spectral theorem hold 
in $ \mathcal{C}_M(\Omega) $; see \cite{BBMT, B, CC, CC2}. The extension result relies on a crucial construction of cutoff functions whose successive derivatives satisfy a certain type of optimal estimates. This construction is due to Bruna \cite{B}; see also {\cite[Proposition 4]{CC}}. Up to a rescaling in the statement of \cite{CC}, the result can be written as follows.

\begin{lem}[\cite{B,CC}]\label{cutoff} 
There is a constant $ c>0 $ such that, for any real numbers $ r>0 $ and $ \sigma>0 $, one can find a function $ \chi_{r,\sigma} $ belonging to $ \mathcal{C}_M(\mathbb{R}^n) $, compactly supported in the ball $B= B(0,r)$, and such that we have $ 0\leq \chi_{r,\sigma} \leq 1 $, $  \chi_{r,\sigma}(t) = 1 $ for $ \vert t\vert\leq r/2 $ and  
$\Vert \chi_{r,\sigma}\Vert_{\overline{B},c\sigma} \leq (h_M(\sigma r))^{-1} $. 
\end{lem}

We shall also need a basic result on flat functions. Given a closed subset $Z$ of $ \Omega$, recall that $\imax_{Z,M}^\infty $ denotes the ideal of functions of $ \mathcal{C}_M(\Omega)$ which are flat at each point of $Z$. 

\begin{lem}\label{flatness}
Let $f$ be an element of $\imax_{Z,M}^\infty $. For any compact subset $K$ of $\Omega$, there are positive constants $ c_1$ and $ c_2$ such that, for any multi-index $I $ in $\mathbb{N}^n$ and any $x$ in $K$, we have
\begin{equation}\label{flatmaj}
\vert D^I f(x)\vert \leq c_1 c_2^i i!M_i h_M(c_2\dist(x,Z)).
\end{equation} 
\end{lem}

\begin{proof} For any real $r>0$, put $K_r=\{y\in \Omega : \dist(y,K)\leq r\} $. If $r$ is chosen small enough, $K_r$ is a compact subset of $\Omega$. Thus, there is a constant $ \sigma>0$ such that, for any $y\in K_r$, $I\in \mathbb{N}^n$ and $J\in \mathbb{N}^n$, we have
$ \vert D^{I+J}f(y)\vert\leq \Vert f\Vert_{K_r,\sigma}\sigma^{i+j}(i+j)! M_{i+j}$. Using \eqref{modg} and the elementary estimate $(i+j)!\leq 2^{i+j}i!j!$, we get 
\begin{equation}\label{normder}
\vert D^{I+J} f(y)\vert\leq c_1c_2^i i!M_i c_2^j j!M_j.
\end{equation}
with $c_1=\Vert f\Vert_{K_r,\sigma}$ and $c_2=2A\sigma$. Now let $x$ be a point in $K$, and let $z$ be a point in $Z$ such that 
\begin{equation}\label{dista2}
\vert x-z\vert=\dist(x,Z).
\end{equation}
If $\dist(x,Z)\leq r $, then the segment $[x,z]$ is contained in $K_r$. Let $j$ be an integer. Since $D^I f$ is flat at $z$, the Taylor formula easily yields $\vert D^If(x)\vert \leq n^j\sup_{\vert J\vert=j,\, y\in K_r} \vert D^{I+J}(y)\vert \vert x-z\vert^j /j! $. Using \eqref{normder} and \eqref{dista2}, and taking the infimum with respect to $j$, we obtain \eqref{flatmaj} up to the replacement of $c_2$ by $nc_2$.  If $\dist(x,Z)> r $, the estimate is a simple consequence of the definition of $\mathcal{C}_M(\Omega)$, up to another modification of $c_1$ and $c_2$. 
\end{proof}

\section{{\L}ojasiewicz ideals}

The following notion will serve as a replacement for the standard {\L}ojasiewicz inequality. 

\begin{defin}\label{Lcond} 
Let $ \varphi$ be a non-zero element of $ \mathcal{C}_M(\Omega)$ and let $X$ be the zero set of $ \varphi$. We say that $\varphi$ \emph{satisfies the $ \mathcal{C}_M $ {\L}ojasiewicz condition} if, for any compact subset $K$ of $\Omega$ and any real $\lambda>0 $, one can find positive constants $ C $ and $ \sigma $ (depending on $ K$ and $ \lambda$) such that, for any multi-index $ J\in\mathbb{N}^n $ and any $ x\in K\setminus X$, we have
\begin{equation}\label{mainestim}
\left\vert D^J(1/\varphi)(x)\right\vert\leq \frac{C \sigma^j j!M_j}{h_M(\lambda \dist(x,X))}.
\end{equation}
\end{defin}

\begin{rem}\label{neighb}
From the basic properties of $ h_M$ in Section \ref{sequences}, we see that, on a given open subset $ \{x\in\Omega : \dist(x,X)>\delta \} $ with $\delta>0$, the $\mathcal{C}_M$ {\L}ojasiewicz condition amounts to nothing more than the conclusion of Lemma \ref{invers}. It is relevant only as a bound on the explosion of $1/\varphi$ and its derivatives in a neighborhood of the zeros of $ \varphi $.  
\end{rem}
 
In Section \ref{additional}, we will provide examples of functions for which the $\mathcal{C}_M$ {\L}ojasiewicz condition holds. Lemma \ref{boundary} below shows that such functions cannot have ``too many flat points'' on the boundary of their zero set.  

\begin{lem}\label{boundary}
Let $ \varphi$ be a non-zero element of $ \mathcal{C}_M(\Omega) $ and let $X$ be its zero set. Assume that $\varphi$ satisfies the $ \mathcal{C}_M$ {\L}ojasiewicz condition, and let $ X_{\infty} = \{a\in X : T_a\varphi= 0\} $ be the set of points of flatness of $ \varphi $. Then $ X\setminus X_{\infty} $ is dense in the boundary $\partial X$ of $ X $. 
\end{lem}

\begin{proof} Notice that $\varphi $ is necessarily flat at each interior point of $X$, hence the inclusion $  X\setminus X_{\infty}\subset \partial X $. We prove the density property by contradiction. If the property is not true, there are a point $ a$ in $ \partial X$ and an open neighborhood $ \omega $ of $a $ in $\Omega $, such that $ \varphi $ is flat on $ \omega\cap\partial X $. Put $ K= \overline{B(a,r)} $ with $ r=\frac{1}{2}\dist(a,X\setminus\omega) $. Then $K$ is a compact subset of $ \omega$ and we have 
\begin{equation}\label{prop1}
\dist(x,\omega\cap\partial X)=\dist(x,\partial X)=\dist(x,X)\textrm{ for any } x \in K.
\end{equation}
Using Lemma \ref{flatness} on the open set $ \omega$, with $ f=\varphi_{\vert\omega} $, $ Z=\omega\cap\partial X $ and $ I=0$, we see that there are constants $c_1$ and $c_2$ such that we have
$\vert\varphi(x)\vert \leq c_1 h_M(c_2\dist(x,\omega\cap \partial X))$ for any $ x \in K$.
Taking property \eqref{hfunct2} into account, we obtain, for any $ x\in K $,
\begin{equation}\label{prop2}
\vert\varphi(x)\vert \leq c_1 h_M(c_3\dist(x,\omega\cap \partial X))^2
\end{equation}
with $c_3=\rho c_2 $. 
On the other hand, using the $\mathcal{C}_M$ {\L}ojasiewicz condition with $\lambda=c_3 $ and $ J=0$, we obtain a constant $ c_4>0 $ such that, for any $ x\in K\setminus X $, 
\begin{equation}\label{prop3}
\vert \varphi(x)\vert\geq c_4 h_M(c_3\dist(x, X)).
\end{equation}
Gathering \eqref{prop1}, \eqref{prop2} and \eqref{prop3}, we obtain $ h_M(c_3 d(x,X))\geq c_4/c_1 $ for any $ x\in K\setminus X $, which is impossible since $ K\setminus X$ has at least an accumulation point on $X$, namely the point $a$. 
\end{proof}

We are now able to state the main result.

\begin{thm}\label{main}
Let $\varphi$ be a non-zero element of $ \mathcal{C}_M(\Omega)$, let $X$ be its zero set, and let $X_{\infty}$ be its set of points of flatness. Put $\mathcal{I}=\varphi \mathcal{C}_M(\Omega)$. The following properties are equivalent:
\begin{enumerate}[{$(A')$}]
\item The function $ \varphi $ satisfies the $ \mathcal{C}_M$ {\L}ojasiewicz condition.
\item  $\imax_{X,M}^\infty\subset \mathcal{I} $ and $ X\setminus X_{\infty} $ is dense in $ \partial X$.
\item $ \imax_{X,M}^\infty=\mathcal{I}\imax_{X,M}^\infty $.
\end{enumerate}
\end{thm}

\begin{proof} We prove the implication $(C')\Rightarrow (A')$ first. We use the (DFS)-space $ \mathcal{D}_M(\Omega)= \varinjlim \mathcal{D}_\nu $ defined in Section \ref{classes}. The intersection $ \mathcal{D}_M(\Omega)\cap \imax_{X,M}^\infty$ is obviously closed in $ \mathcal{D}_M(\Omega)$, hence it is also a (DFS)-space with step spaces $ \mathcal{E}_\nu = \mathcal{D}_\nu\cap  \imax_{X,M}^\infty$. 

It is easy to see that the map $ \Lambda : \mathcal{D}_M(\Omega)\cap\imax_{X,M}^\infty \rightarrow \mathcal{D}_M(\Omega)\cap\imax_{X,M}^\infty $ defined by $ \Lambda (f)=\varphi f $ is continuous. Moreover, given an element $ g$ of $ \mathcal{D}_M(\Omega)\cap\imax_{X,M}^\infty  $, the assumption implies that it can be written $ \varphi h $ for some $ h\in \imax_{X,M}^\infty $. If $ \chi$ is an element of $ \mathcal{D}_M(\Omega) $ such that $ \chi=1 $ on $\supp g $, we then have $ g=\chi g=\varphi f $ with $ f=\chi h \in \mathcal{D}_M(\Omega)\cap\imax_{X,M}^\infty $. Thus, $ \Lambda $ is also surjective.  

We can therefore apply the De Wilde open mapping theorem ({\cite[Chapter 24]{MV}}), which yields the following property: for any $ \nu \geq 1 $, there exist an integer $ \mu_\nu\geq 1 $ and a real constant $ C_\nu>0 $ such that, for any $ g\in \mathcal{E}_\nu $, one can find an element $ f$ of $ \mathcal{E}_{\mu_\nu}$ such that 
\begin{equation}\label{opmap}
\varphi f=g\ \ \textrm{and}\ \ \Vert f\Vert_{\mu_\nu}\leq C_\nu \Vert g\Vert_\nu.  
\end{equation}
 
Now, let $x $ be a point in $ K\setminus X$, let $ d_K $ be a real number such that
$0<d_K< \dist(K,\mathbb{R}^n\setminus \Omega)$, and put $r_x= \min(\dist(x,X),d_K)$. Given $\lambda>0 $, we apply Lemma \ref{cutoff} with $r=2r_x/3$ and $\sigma=3\lambda/2 $. We set $ g_x(y)=\chi_{r,\sigma}(y-x) $. Then $ g_x$ belongs to $ \mathcal{C}_M(\Omega) $ and is compactly supported in the ball $B_x=B(x,2r_x/3)$. Obviously
$ B_x $ is contained in $ K'=\{y \in \Omega : \dist(y,K)\leq 2d_K/3\} $, which is a compact subset of $\Omega$. For a sufficiently large integer $ \nu$, depending only on $K $ and $ \lambda$, we have $ \nu\geq c\sigma $ and $ K'\subset K_\nu$, so that $ g_x$ belongs to $\mathcal{E}_\nu $ and
\begin{equation}\label{eq1}
 \Vert g_x\Vert_\nu= \Vert g_x\Vert_{\overline{B_x},\nu}
 \leq \Vert g_x\Vert_{\overline{B_x},c\sigma}
\leq \left(h_M(\lambda r_x)\right)^{-1}.
\end{equation}
Since $ h_M(\lambda r_x)$ equals either $ h_M(\lambda\dist(x,X)) $ or $ h_M(\lambda d_K)$, and since we have $ h_M(t)\leq 1 $ for every $ t>0$, we see that 
\begin{equation}\label{eq2}
h_M(\lambda r_x)\geq h_M(\lambda d_K) h_M(\lambda \dist(x,X)). 
\end{equation}
Now, if $ f_x $ denotes the element of $ \mathcal{E}_{\mu_\nu} $ associated with $ g_x$ by property \eqref{opmap}, we therefore have $ \varphi f_x=g_x $ and, thanks to \eqref{eq1} and \eqref{eq2},  
\begin{equation}\label{fin1}
\Vert f_x\Vert_{\mu_\nu}\leq C'_\nu \left(h_M(\lambda \dist(x,X))\right)^{-1}
\end{equation}
with $C'_\nu=C_\nu/h_M(\lambda d_K)$. 
For any $y $ in $ B'_x=B(x,r_x/3) $, we have $ g_x(y)=1 $, hence
\begin{equation}\label{fin2}
 f_x(y)= 1/\varphi(y).
\end{equation}
In particular, we have $ f_x(y)\neq 0 $. Thus, we derive $ B'_x\subset \supp f_x\subset K_{\mu_\nu} $, which implies, for any $ y\in B'_x $ and any multi-index $J$, 
\begin{equation}\label{fin3}
\vert D^J f_x(y)\vert \leq \Vert f_x\Vert_{\mu_\nu} (\mu_\nu)^j j!M_j.
\end{equation}
Combining \eqref{fin1}, \eqref{fin2} and \eqref{fin3}, we get the desired estimate \eqref{mainestim} with suitable constants $A=C'_\nu$ and $ B=\mu_\nu$ depending only on $ \nu$, hence only on $ K $ and $\lambda$.\\

We now prove the implication $(A')\Rightarrow (B')$. By Lemma \ref{boundary}, the assumption implies that $ X\setminus X_\infty $ is dense in $ \partial X $. The proof of the inclusion $ \imax_{X,M}^\infty\subset\mathcal{I}$ is a variant of the proof of {\cite[Theorem 2.3]{Th1}}; we give some details for the reader's convenience. Let $f$ be an element of $ \imax_{X,M}^\infty $. For any $ x\in \Omega\setminus X $ and any multi-index $P\in\mathbb{N}^n $, the Leibniz formula yields
\begin{equation}\label{sum}
D^P(f/\varphi)(x)=\sum_{I+J=P}\frac{P!}{I!J!}D^If(x)D^J(1/\varphi)(x).
\end{equation} 
Let $K$ be a compact subset of $ \Omega$. For $  x\in K\setminus X$, we combine the $ \mathcal{C}_M$ {\L}ojasiewicz condition with Lemma \ref{flatness} in order to obtain an estimate for all the terms $D^If(x)D^J(1/\varphi)(x)$ that appear in \eqref{sum}. 
Lemma \ref{flatness}, together with \eqref{hfunct2}, yields $ \vert D^I f(x)\vert \leq c_1 c_2^i i!M_i(h_M(c_3\dist(x,X)))^2$ with $ c_3=\rho c_2 $. Applying the $ \mathcal{C}_M$ {\L}ojasiewicz condition with $ \lambda=c_3 $, we therefore get $ \vert D^If(x)D^J(1/\varphi)(x)\vert\leq c_2Cc_2^i \sigma^j i!j!M_iM_jh_M(c_3\dist(x,X)) $. Since $ i+j=p$, we have $ i!j!\leq p! $, as well as $ M_iM_j\leq M_p $ by \eqref{logprod}. Inserting these estimates in \eqref{sum}, we obtain, for every multi-index $P$ and every $x\in K\setminus X $,
\begin{equation}\label{quotient}
\left\vert D^P(f/\varphi)(x)\right\vert\leq c_5c_6^p p!M_p h_M(c_3\dist(x,X))
\end{equation}
with $c_5=c_2C$ and $c_6=c_2+\sigma $. Using \eqref{quotient} and the Hestenes lemma, we see that the function $ g $ defined by $ g(x)=f(x)/\varphi(x) $ for $ x\in \Omega\setminus X $ and $g(x)=0 $ for $ x\in X$ belongs to $ \mathcal{C}_M(\Omega)$. Obviously, we have $ f=\varphi g $, hence $ f\in \mathcal{I}$.\\ 

Finally, we prove the implication $(B')\Rightarrow (C')$. Let $ f$ be an element of $ \imax_{X,M}^\infty $. By assumption, there is $ g\in \mathcal{C}_M(\Omega)$ such that $ f=\varphi g$. Let $ a $ be a point of $ X\setminus X_\infty $. In the ring of formal power series, we have $ 0=T_a f=(T_a\varphi)(T_ag) $ with $T_a\varphi\neq 0$, which implies $ T_ag=0 $. Thus, $ g $ is flat on $ X\setminus X_\infty $, hence on $ \partial X $ since it is assumed that $ X\setminus X_\infty $ is dense in $ \partial X$. Put $ \tilde{g}(x)=g(x) $ for $ x\in \Omega\setminus X $ and $ \tilde{g}(x)=0 $ for $ x\in X $. By the Hestenes lemma, it is then readily seen that $ \tilde{g} \in \imax^\infty_{X,M}$. Moreover, we have $ f=\varphi\tilde{g} $, hence $ f \in \mathcal{I}\imax_{X,M}^\infty $, and the proof is complete. 
\end{proof}

\begin{rem} We do not know whether the implication $(B')\Rightarrow (C')$ still holds without the additional assumption on $ X\setminus X_\infty$ in $(B')$. This is true when $X$ is a real-analytic submanifold of $\Omega$: indeed, according to {\cite[Theorem 4.2.4]{Th4}}, we then have\footnote{The result in \cite{Th4} is actually a local version of that statement, but it can be globalized, using partitions of unity.} $\imax_{X,M}^\infty=\imax_{X,M}^\infty\imax_{X,M}^\infty $. Thus, in this case, the inclusion $ \imax_{X,M}^\infty\subset \mathcal{I} $ easily implies $(C')$. 
\end{rem}

\begin{rem}\label{indepgen} 
Using the equivalence $(A')\Leftrightarrow (C')$, we see that if $ \varphi$ satisfies the $\mathcal{C}_M$ {\L}ojasiewicz condition and if $ h $ is an invertible element of the algebra $ \mathcal{C}_M(\Omega)$, so that $\varphi$ and $h\varphi$ generate the same ideal $\mathcal{I}$, then $ h\varphi$ also satisfies the $\mathcal{C}_M$ {\L}ojasiewicz condition. This can also be checked by a direct computation with the Leibniz formula.  
\end{rem}

\section{Additional properties and examples}\label{additional}

\subsection{On the zero set} 

We have a Denjoy-Carleman counterpart of {\cite[Proposition V.4.6]{Tou}}.

\begin{prop} Let $\varphi$ be an element of $\mathcal{C}_M(\Omega)$ that satisfies the $\mathcal{C}_M$ {\L}ojasiewicz condition, and let $X$ be its zero set. Then there is a $\mathcal{C}_M$-smooth submanifold $Y$ of $\Omega$ such that $X=\overline{Y}$.
\end{prop}
 
\begin{proof} We notice first that the conclusion of Lemma \ref{boundary} only requires a weaker property than the $\mathcal{C}_M$ {\L}ojasiewicz condition: more precisely, the proof remains valid as soon as, for any compact subset $K$ of $\Omega$ and any real $\lambda>0$, one can find a constant $C>0$ such that the inequality
$ \vert \varphi(x)\vert\geq C h_M(\lambda \dist(x,X))$ holds for any $x\in K$. It is then fairly easy to check that the proof by induction given in \cite{Tou} for the usual {\L}ojasiewicz inequality on $\mathcal{C}^\infty$ functions remains valid in the $\mathcal{C}_M$ case, up to minor modifications. 
\end{proof}

\subsection{Connection with closedness} In this section, we show that the $ \mathcal{C}_M$ {\L}ojasiewicz condition behaves as expected with respect to closedness properties of ideals. 

\begin{prop}\label{connectclosed} Let $ \varphi $ be a non-zero element of $ \mathcal{C}_M(\Omega) $ that generates a closed ideal in $ \mathcal{C}_M(\Omega)$. Then $ \varphi$ satisfies the $ \mathcal{C}_M$ {\L}ojasiewicz condition. Moreover, both properties are equivalent when the zeros of $\varphi$ are isolated.  
\end{prop}

\begin{proof} We use the same notation as in the proof of the implication $(C')\Rightarrow (A')$ of Theorem \ref{main}. Put $ \mathcal{I}=\varphi \mathcal{C}_M(\Omega) $ and assume that $\mathcal{I}$ is closed in $ \mathcal{C}_M(\Omega)$. Since the inclusion $ \mathcal{D}_M(\Omega)\hookrightarrow \mathcal{C}_M(\Omega)$ is continuous, $ \mathcal{I}\cap \mathcal{D}_M(\Omega)$ is closed in $ \mathcal{D}_M(\Omega)$. Using cutoff functions, it is also easy to see that $ \mathcal{I}\cap \mathcal{D}_M(\Omega)=\varphi \mathcal{D}_M(\Omega)$. It is then possible to duplicate the proof of the implication $(C')\Rightarrow (A')$, the only difference being that the map $ f\mapsto \varphi f $ is now considered as a map from the (DFS)-space $ \mathcal{D}_M(\Omega)$ onto its closed subspace $ \varphi \mathcal{D}_M(\Omega)$. 

The converse in the case of isolated zeros is based on a variant of the argument leading to {\cite[Proposition 4.1]{Th1}} (which deals with a singleton). Assume that $\varphi$ satisfies the $\mathcal{C}_M$ {\L}ojasiewicz condition and that its zero set $X$ consists of isolated points, so that $X$ is a countable subset $ \{a_j :j\geq 1\}$ of $\Omega$. Put $\mathcal{I}=\varphi \mathcal{C}_M(\Omega)$ and let $f$ be an element of the closure $\overline{\mathcal{I}}$. By the $\mathcal{C}_M$ version of Whitney's spectral theorem \cite{CC2}, for every $j\geq 1$ there is a function $g_j$ of $ \mathcal{C}_M(\Omega)$ such that $ f-\varphi g_j $ is flat at $a_j$. Let $ (\chi_j)_{j\geq 1} $ be a sequence of compactly supported elements of $ \mathcal{C}_M(\Omega)$  such that $ \chi_j=1 $ in a neigborhood of $a_j$ and $ \supp{\chi_j}\cap \supp{\chi_k}=\emptyset $ for $ k\neq j$. Then the (locally finite) series $g= \sum_{j\geq 1}\chi_j g_j $ defines an element of $\mathcal{C}_M(\Omega)$ and we have $ f-\varphi g \in \imax_{X,M}^\infty $. Since $(B')$ holds, this yields $f\in \mathcal{I}$, hence the result.
\end{proof}

\begin{exam}\label{appli}
According to Proposition \ref{connectclosed} and the results in \cite{Th1, Th3}, examples of functions $ \varphi$ which satisfy the $ \mathcal{C}_M$ {\L}ojasiewicz condition will include any homogeneous polynomial with an isolated real critical point at $0$, as well as real analytic functions whose germs of complex zeros intersect $\mathbb{R}^n$ at isolated points with {\L}ojasiewicz exponent $1$ for the regular separation property. 
On the other hand, some analytic functions do not satisfy the $ \mathcal{C}_M$ {\L}ojasiewicz condition: for instance, given an integer $k\geq 2$, the polynomial $ \psi(x)=x_1^2+x_2^{2k} $ does not satisfy the $\mathcal{C}_M $ {\L}ojasiewicz condition in $\mathbb{R}^2$, as can be seen from the results in \cite{Th1} (property $(B')$ fails).  
\end{exam}

We now give an example showing that the converse to Proposition \ref{connectclosed} is false without the assumption of isolated zeros. In particular, the $ \mathcal{C}_M$ {\L}ojasiewicz condition does not imply closedness in general. 

\begin{exam} We put $n=2$, $\Omega=\mathbb{R}^2$, and $\varphi(x)= x_1 \psi(x) $ where $ \psi$ is the polynomial mentioned in Example \ref{appli}. We then have $ X=\{x\in \mathbb{R}^2 : x_1=0 \} $ and $\dist(x,X)=\vert x_1\vert$. Let $x$ be a point in $ \mathbb{R}^2\setminus X $. For any $ v=(v_1,v_2)\in \mathbb{C}^2 $, we have
\begin{equation*}
\vert \psi(x+v)-\psi(x)\vert\leq 2 \vert x_1\vert\vert v_1\vert +\vert v_1\vert^2 + \sum_{p=1}^{2k}\binom{2k}{p}\vert x_2\vert^{2k-p}\vert v_2\vert^p.
\end{equation*}
We also have the obvious inequalities $ \vert x_1\vert\leq (\psi(x))^{1/2} $ and $ \vert x_2\vert\leq (\psi(x))^{1/2k}$. Thus, if we assume $\vert v_1\vert \leq \delta (\psi(x))^{1/2} $ and $ \vert v_2\vert\leq \delta (\psi(x))^{1/2k} $ for some real number $ \delta $ with $0<\delta<1$, we get
\begin{equation*}
\vert \psi(x+v)-\psi(x)\vert\leq \left(2\delta + \delta^2+ \sum_{p=1}^{2k}\binom{2k}{p}\delta^p\right)\psi(x) \leq (2^{2k}+2)\delta \psi(x).
\end{equation*}
Setting $ \delta = (2^{2k+1}+4)^{-1} $, we obtain $\vert \psi(\zeta)\vert\geq \frac{1}{2}\psi(x) $ for every point $ \zeta $ in the bidisc $ \{ \zeta\in\mathbb{C}^2 : \vert \zeta_1-x_1\vert \leq \delta (\psi(x))^{1/2},\ \vert \zeta_2-x_2\vert\leq \delta (\psi(x))^{1/2k}\} $. The Cauchy formula then yields, for every $(i,j)\in\mathbb{N}^2 $,
\begin{equation*}
\left\vert \frac{\partial^{i+j}}{\partial x_1^i x_2^j}\left( \frac{1}{\psi(x)}\right)\right\vert\leq 2\delta^{-(i+j)}i!j! (\psi(x))^{-(\frac{i}{2}+\frac{j}{2k}+1)},
\end{equation*}
which easily implies
\begin{equation}\label{psi}
\left\vert \frac{\partial^{i+j}}{\partial x_1^i x_2^j} \left(\frac{1}{\psi(x)}\right)\right\vert\leq 2\delta^{-(i+j)}i!j! \vert x_1\vert^{-(i+j+2)}
\end{equation}
provided we assume $\vert x_1\vert < 1$. Using \eqref{psi}, the definition of $ \varphi $, and the Leibniz formula, we then get
\begin{equation*}
\left\vert \frac{\partial^{i+j}}{\partial x_1^i x_2^j} \left(\frac{1}{\varphi(x)}\right)\right\vert \leq B^{i+j+1} i!j! \vert x_1\vert^{-(i+j+2)}
\end{equation*}
for some suitable constant $ B>0 $. Let $ \lambda $ be a given positive real number. We write $ \vert x_1\vert^{-(i+j+2)}=\frac{\lambda^{i+j+2} M_{i+j+2}}{(\lambda \vert x_1\vert)^{i+j+2}M_{i+j+2}} $. The definition of $h_M$ implies $ (\lambda \vert x_1\vert)^{i+j+2}M_{i+j+2} \geq h_M(\lambda \vert x_1\vert) = h_M(\lambda\dist(x,X)) $, whereas \eqref{modg} yields $ M_{i+j+2}\leq A^{i+j+2}M_2 M_{i+j} $. Gathering these inequalities, we eventually obtain 
$\vert x_1\vert^{-(i+j+2)}\leq (A\lambda)^{i+j+2} (h_M(\lambda \dist(x,X)))^{-1} $ and
\begin{equation*}
\left\vert \frac{\partial^{i+j}}{\partial x_1^i x_2^j} \left(\frac{1}{\varphi(x)}\right)\right\vert \leq \frac{C\sigma^{i+j} (i+j)! M_{i+j}}{h_M(\lambda \dist(x,X))} 
\end{equation*}
with $ C= A^2B\lambda^2 $ and $\sigma= AB\lambda $. Thus, we have established the desired estimate for $ \vert x_1\vert=\dist(x,X)<1 $, which suffices to conclude that $\varphi $ satisfies the $\mathcal{C}_M$ {\L}ojasiewicz condition (see Remark \ref{neighb}). 
However, the ideal $ \mathcal{I}= \varphi \mathcal{C}_M(\mathbb{R}^2) $ is not closed for $ k\geq 2 $. Indeed, in this case, it has been shown in \cite{Th1} that the ideal $ \mathcal{J}=\psi \mathcal{C}_M(\mathbb{R}^2) $ is not closed. Since $\mathcal{J}$ is the preimage of $\mathcal{I} $ under the continuous mapping $ \Pi : \mathcal{C}_M(\mathbb{R}^2)\to \mathcal{C}_M(\mathbb{R}^2) $ defined by 
$ \Pi(f)(x)=x_1 f(x) $, we see that $ \mathcal{I} $ is not closed either. 
\end{exam}

We conclude with a natural question. 

\begin{prob} Is it possible to extend the above results to the general case of finitely generated ideals? A first idea is to mimic the definition of {\L}ojasiewicz ideals in the $\mathcal{C}^\infty$ case, and say that a finitely generated ideal of $\mathcal{C}_M(\Omega)$ is {\L}ojasiewicz if it contains an element $\varphi$ which satisfies the $\mathcal{C}_M$ {\L}ojasiewicz condition. However, this definition doesn't seem to allow an immediate extension of the crucial implication $(C')\Rightarrow (A')$, whose proof is quite different from the $\mathcal{C}^\infty$ case and doesn't seem easily adaptable to the case of several generators.
\end{prob}


\begin{thebibliography}{15}

%%%%%%% To ease editing, use normal size:

\normalsize
\baselineskip=17pt

%%%%%%%%%%%%%%%

\bibitem{BBMT} J. Bonet; R. Braun; R. Meise; B.A. Taylor, \emph{Whitney's extension theorem for nonquasianalytic classes of ultradifferentiable functions}, Studia Math. {99} (1991), 155--184.

\bibitem{B} J. Bruna, \emph{An extension theorem of Whitney type for non quasianalytic classes of functions}, J. London Math. Soc. {22} (1980), 495--505.

\bibitem{CC} J. Chaumat; A.-M. Chollet, \emph{Surjectivit\'e de l'application restriction \`a un compact dans des classes de fonctions ultradiff\'erentiables}, Math. Ann. {298} (1994), 7--40.

\bibitem{CC2} J. Chaumat; A.-M. Chollet, \emph{Caract\'erisation des anneaux noetheriens de s\'eries formelles \`a croissance contr\^ol\'ee. Application \`a la synth\`ese spectrale}, Publ. Math. {41} (1997), 545--561.

\bibitem{Kom} H. Komatsu, \emph{Ultradistributions, I. Structure theorems and a characterization}, J. Fac. Sci. Univ. Tokyo, Sect. IA {20} (1973), 25--105.

\bibitem{Kom2} H. Komatsu, \emph{The implicit function theorem for ultradifferentiable mappings}, Proc. Japan Acad. ser. A {55} (1979), 69--72.

\bibitem{MV} R.Meise; D. Vogt, \emph{Introduction to functional analysis}, The Clarendon Press, Oxford University Press, New York, 1997.

\bibitem{Ris} J.-J. Risler, \emph{Le th\'eor\`eme des z\'eros pour les id\'eaux de fonctions diff\'erentiables en dimension $2$ et $3$}, Ann. Inst. Fourier {26} (1976), 73--107.

\bibitem{Rou} C. Roumieu, \emph{Ultradistributions d\'efinies sur $\mathbb{R}^n$ et sur certaines classes de vari\'et\'es diff\'erentiables}, J. Analyse Math. {10} (1962-1963), 153--192. 

\bibitem{Th1} V. Thilliez, \emph{On closed ideals in smooth classes}, Math. Nachr. {227} (2001), 143--157.

\bibitem{Th2} V. Thilliez, \emph{A sharp division estimate for ultradifferentiable germs}, Pacific J. Math. {205} (2002), 237--256.  

\bibitem{Th3} V. Thilliez, \emph{Bounds for quotients in rings of formal power series with growth constraints}, Studia Math. {151} (2002), 49--65

\bibitem{Th4} V. Thilliez, \emph{Division by flat ultradifferentiable functions and sectorial extensions}, Results Math. {44} (2003), 169--188

\bibitem{Tho} R. Thom, \emph{On some ideals of differentiable functions}, J. Math. Soc. Japan {19} (1967), 255--259

\bibitem{Tou} J.-C. Tougeron, \emph{Id\'eaux de fonctions diff\'erentiables}, Springer Verlag, Berlin, 1972. 

\end{thebibliography}
\end{document}